\documentclass[11pt]{article}
\usepackage{amsfonts,amsmath,amssymb,amsthm, amscd, subcaption}
\usepackage{float}
\usepackage{graphicx}
\numberwithin{equation}{section}

\newtheorem{theorem}{Theorem}[section]
\newtheorem{prop}[theorem]{Proposition}
\newtheorem{lemma}[theorem]{Lemma}

\theoremstyle{definition}
\newtheorem{definition}[theorem]{Definition}
\newtheorem{example}[theorem]{Example}
\newtheorem{remark}[theorem]{Remark}

\def\<{{\langle}}
\def\>{{\rangle}}

\def\g{{\gamma}}
\def\d{{\delta}}
\def\Z{\mathbb Z}

\def\R{\mathbb R}

\def\S{{\mathbb S}}

\def\C{{\cal C}}

\def\D{{\cal D}}

\def\L{{\cal L}}
\def\La{{\Lambda}}

\def\e{\epsilon}

\def\De{\Delta}

\def\ni{\noindent} 

\begin{document}

\title{Group Presentations for Links in Thickened Surfaces  }

\author{Daniel S. Silver 
\and Susan G. Williams}

\maketitle 


\begin{abstract} Using a combinatorial argument, we prove the well-known result that the Wirtinger and Dehn presentations of a link in 3-space describe isomorphic groups. The result is not true for links $\ell$ in a thickened surface $S \times [0,1]$. Their precise relationship, as given in \cite{By12}, is established here by an elementary argument. When a  diagram in $S$ for $\ell$ can be checkerboard shaded, the Dehn presentation leads naturally to an abelian ``Dehn coloring group," an isotopy invariant of $\ell$. Introducing homological information from $S$ produces a stronger invariant, $\cal C$, a module over the group ring of $H_1(S; {\mathbb Z})$. The authors previously defined the Laplacian modules ${\cal L}_G,{ \cal L}_{G^*}$  and polynomials $\Delta_G, \Delta_{G^*}$ associated to a Tait graph $G$ and its dual $G^*$, and showed that the pairs $\{{\cal L}_G, {\cal L}_{G^*}\}$, $\{\Delta_G, \Delta_{G^*}\}$ are isotopy invariants of $\ell$. The relationship between $\cal C$ and the Laplacian modules is described and used to prove that $\Delta_G$ and  $\Delta_{G^*}$ are equal when $S$ is a torus.

 \bigskip

\end{abstract}

\section{Introduction} Modern knot theory, which began in the early 1900's, was propelled by the nearly simultaneous publications of two different methods for computing presentations of knot groups, fundamental groups of knot complements. The methods are due to W. Wirtinger and M. Dehn. Both are combinatorial, beginning with a 2-dimensional drawing, or ``diagram," of a knot or link, and reading from it a group presentation.

Of course, Wirtinger and Dehn presentations describe the same group. However, the proof usually involves algebraic topology. Continuing in the combinatorial spirit of 
early knot theory, a spirit that has revived greatly since 1985 with the landmark discoveries of V.F.R. Jones \cite{Jo85}, we offer a diagrammatic proof that the two presentations describe the same group. We then extend our technique to knots and links in thickened surfaces. There the presentations describe different groups. We explain their relationship. 

Diagrams of knots and links in surfaces that can be ``checkerboard shaded" carry the same information as $\pm 1$-weighted embedded graphs. Laplacian matrices of graphs, well known to combinatorists, can be used to describe algebraic invariants that we show are closely related to Dehn presentations. 

The first two sections are relatively elementary and should be accessible to a reader with a basic undergraduate mathematics background. Later sections become more sophisticated but require only modest knowledge of modules. 

The authors are grateful to Louis Kauffman for helpful comments, and also Seiichi Kamada for sharing his and Naoko Kamada's early ideas about Wirtinger and Dehn presentations.

\section{Wirtinger and Dehn Link Group Presentations} \label{Intro} 

A link in $\R^3$ is a finite embedded collection of circles $\ell$ regarded up to ambient isotopy. (A \textit{knot} is a special case, a link with a single connected component.) A link is usually described by a \textit{link diagram}, a 4-valent graph embedded in the plane with each vertex replaced by a broken line segment to indicate how the link passes over itself. Following \cite{Ka06} we call the graph a \textit{universe} of $\ell$. 

Two links are isotopic if and only if a diagram of one can be changed into a diagram of the other by a finite sequence of local modifications called Reidemeister moves, as in Figure \ref{reid}, as well as deformations of the diagram that don't create or destroy crossings. (For a proof of this as well as other well-known facts about links see \cite{Mu96}.) The topological task of determining when two links are isotopic now becomes a combinatorial problem of understanding when two link diagrams are equivalent. Moreover, we can use  Reidemeister moves to discover link invariants; they are quantities associated to a diagram that are unchanged by each of the three moves.

\begin{figure}[H]
\begin{center}
\includegraphics[height=1.5 in]{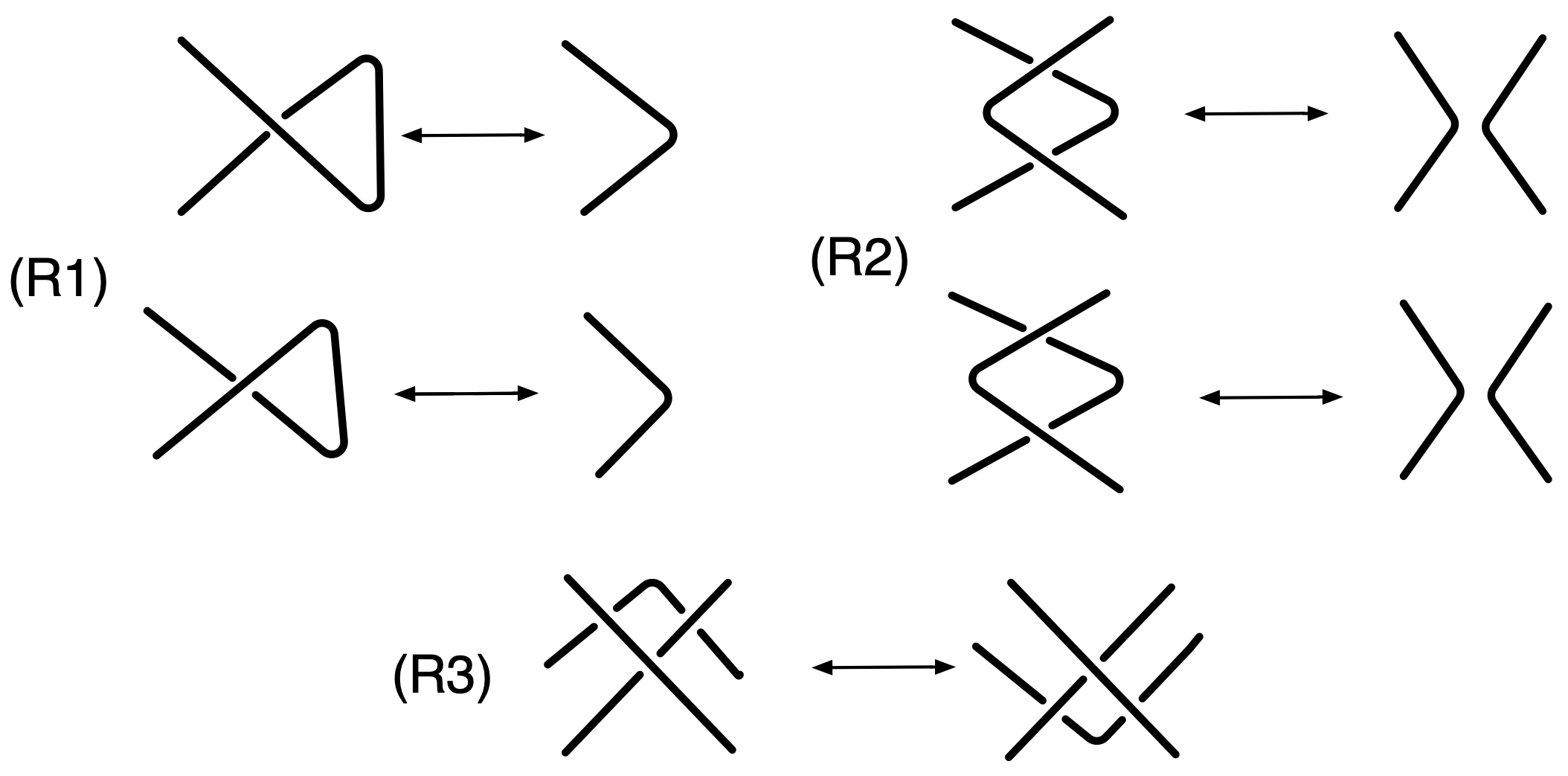}
\caption{Reidemeister moves}
\label{reid}
\end{center}
\end{figure}

The \textit{link group}, the fundamental group $\pi_1(\R^3\setminus \ell)$ of the link complement, is a familiar link invariant. Usually it is described by a group presentation based on a link diagram, the most common being the Wirtinger and the Dehn presentations. In a Wirtinger presentation, which requires that the link be oriented, generators correspond to arcs, maximally connected components of the diagram, while relations correspond to the crossings, as in Figure \ref{wirt}(i). 

We remind the reader that a presentation of a group $\pi$ is an expression of the form $\langle x_1, \ldots, x_n \mid r_1, \ldots, r_m\rangle$, where $x_1, \ldots, x_n$ generate $\pi$ while $r_1, \ldots, r_m$ are \textit{relators}, words in $x_1^{\pm 1}, \ldots, x_n^{\pm 1}$ that represent trivial elements. The group relators are sufficient to describe the group, in the sense that $\pi \cong F/R$, where $F$ is the free group on $x_1, \ldots, x_n$ and $R$ is the normal subgroup of $F$ generated by $r_1, \ldots, r_m$. A group that has such a presentation is said to be \textit{finitely presented}. Often it is more natural to include \textit{relations}, expressions of the form $r=s$, in a presentation rather than relators. Such an expression is another way of writing the relator $rs^{-1}$. 

Just as link diagrams and Reidemeister moves convert the recognition problem for links to a combinatorial task, group presentations and \textit{Tietze transformations} turn the recognition problem for groups into a combinatorial one. Two finitely presented groups are isomorphic if and only one presentation can be converted into the other by a finite sequence of $(T1)^{\pm 1}$ generator addition/deletion and $(T2)^{\pm 1}$ relator addition\ deletion as well as changing the order of the generators or the relators.  

\begin{itemize}
\item{} $(T1): \langle x_1, \ldots, x_n \mid r_1, \ldots, r_m\rangle$ $\to$
$\langle y, x_1, \ldots, x_n \mid yw^{-1}, r_1, \ldots, r_m\rangle$, where $w$ is a word in $x_1^{\pm 1}, \ldots, x_n^{\pm 1}$.

\item{} $(T1)^{-1}$: reverse of $(T1)$, replacing $y$ by $w$ where it appears in $r_1, \ldots, r_m$.

\item{} $(T2): \langle x_1, \ldots, x_n \mid r_1, \ldots, r_m\rangle$ $\to$
$\langle x_1, \ldots, x_n \mid r, r_1, \ldots, r_m\rangle$, where $r$ is a redundant relation (that is, $r \in R$). 

\item{} $(T2)^{-1}$: reverse of $(T2)$.

\end{itemize}

A Dehn presentation ignores the link orientation. Its generators are regions of a diagram, components of the complement of the universe, with one region arbitrarily designated as the \textit{base region} and set equal to the identity. Relations again correspond to crossings, as in Figure \ref{wirt}(ii). The reader can check that the two presentations resulting from 2(ii a) and 2(ii b) describe isomorphic groups via an isomorphism that maps generators to their inverses. Neither depends on of the choice of base region $R_0$ (see Remark \ref{DWremark} belowre). We use the second presentation throughout.

For the sake of simplicity, we will not distinguish between arcs of $\D$ and Wirtinger generators, using the same symbols for both. Similarly, regions of $\D$ will be identified with Dehn generators.

\begin{figure}
\begin{center}
\includegraphics[height=1.5 in]{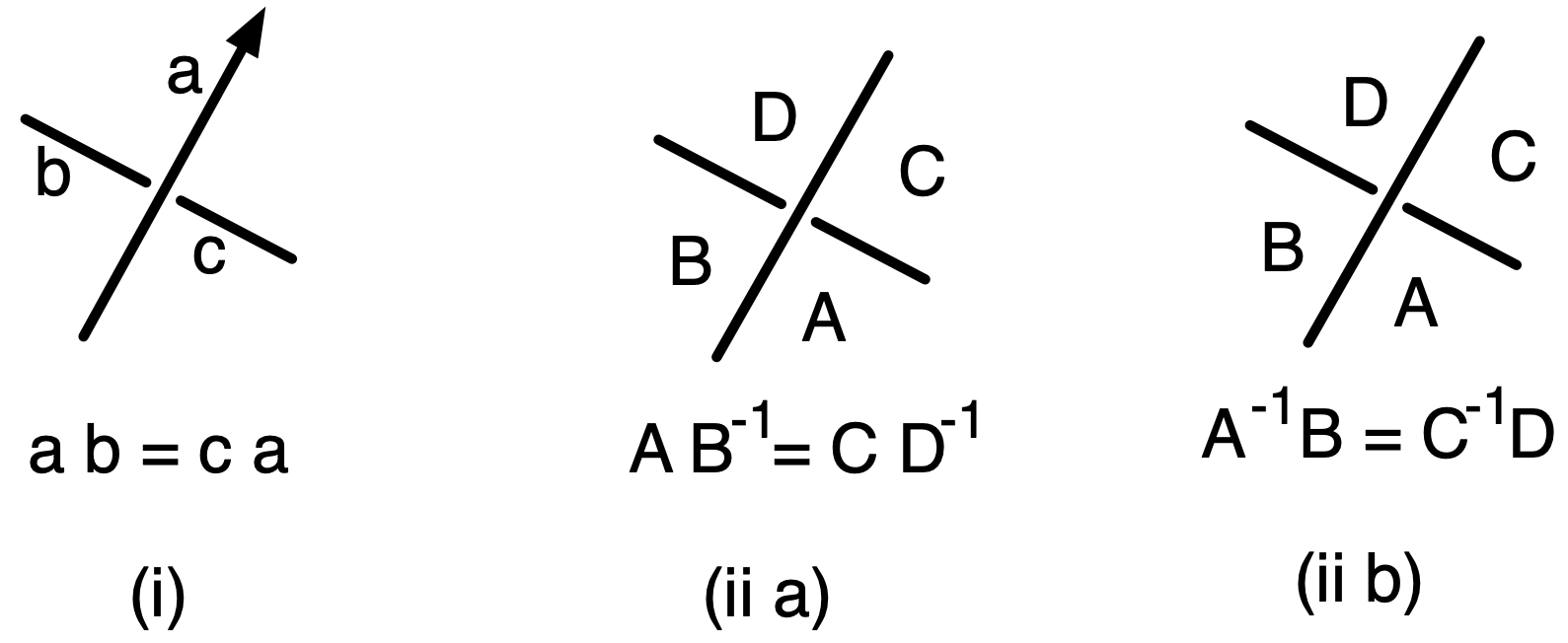}
\caption{(i) Wirtinger relation; (ii) two conventions for Dehn relations}
\label{wirt}
\end{center}
\end{figure}

The group $\pi_{wirt}$ described by the Wirtinger presentation is usually seen to be isomorphic to the link group by a topological argument (see \cite{St80}, for example). Then one proves that the group $\pi_{dehn}$ described by the Dehn presentation is isomorphic to $\pi_{wirt}$ by another topological argument (see \cite{LS77}). In the next section we present a short, purely combinatorial proof that $\pi_{wirt}$ and $\pi_{dehn}$ are isomorphic. The method involves combinatorial ``differentiation" and ``integration" on link diagrams, introduced in \cite{STW18}. Using it we will extend our study to links in thickened surfaces. 

Instead of viewing a link diagram in the plane, we can put it in the 2-sphere $\S^2$. In this egalitarian approach all regions are compact. Such a diagram represents a link in the thickened sphere $\S^2 \times [0,1]$, $\S^3$, or again in $\R^3$.  Regardless of which we choose, two links remain isotopic if and only if their diagrams are transformable into each other by finitely many Reidemeister moves. 

It is natural to replace the 2-sphere by an arbitrary closed, connected orientable surface $S$. A diagram in $S$ represents a link in the thickened surface $S \times [0,1]$. As before, we regard links up to ambient isotopy. Again, two links are isotopic if and only if any diagram of one can be transformed to any diagram of the other by a finite sequence of Reidemeister moves. As explained in \cite{Bo08}, this follows from \cite{HZ64}, which ensures that isotopic links are in fact isotopic by linear moves in arbitrarily small neighborhoods. 

Given a diagram in $S$ for a link $\ell$, the groups $\pi_{wirt}, \pi_{dehn}$ described by the Wirtinger and Dehn presentations are seen to be invariants using Reidemeister moves, but they no longer need be isomorphic. We will describe their precise relationship using combinatorial  integration and differentiation on the diagram.  (For a discussion of the fundamental group $\pi_1(\S \times [0,1]\setminus \ell)$ of the link complement see \cite{CSW14}.)

\section{Integration on Link Diagrams} \label{IntDiff}

There is a natural homomorphism $f_{wd}: \pi_{wirt}\to \pi_{dehn}$, defined  first on generators of $\pi_{wirt}$ and then extended to arbitrary words in the usual way. For any generator $a$ we define $f_{wd}(a)$ to be $A^{-1}B$, where $A$ is the region to the right of the oriented arc $a$, and $B$ is on the other side (Figure \ref{WD}(i)). This is well defined, since if the arc $a$ 
separates another pair of regions $C, D$, as in Figure \ref{WD} (ii), then $A^{-1}B = C^{-1}D$ in $\pi_{dehn}$.
(We think of $A^{-1}B$ as a \textit{derivative} across the arc of our Dehn generator-labeled diagram.) 

We extend $f_{wd}$ in the usual way to a function on words in Wirtinger generators and their inverses. 
In order to show that this induces a homomorphism on $\pi_{wirt}$, we must show that it sends Wirtinger relations to the identity element of $\pi_{dehn}$. For this consider a Wirtinger relation as in Figure \ref{WD}(ii). It is mapped by $f_{wd}$ to $f_{wd}(ab) = f_{wd}(ca)$, which can be written $(C^{-1}D)(D^{-1}B) = (C^{-1}A)(A^{-1}B)$. This simplifies to $C^{-1}B= C^{-1}B$. The case of a left-hand crossing is similar. 

\begin{figure}[H]
\begin{center}
\includegraphics[height=1.5 in]{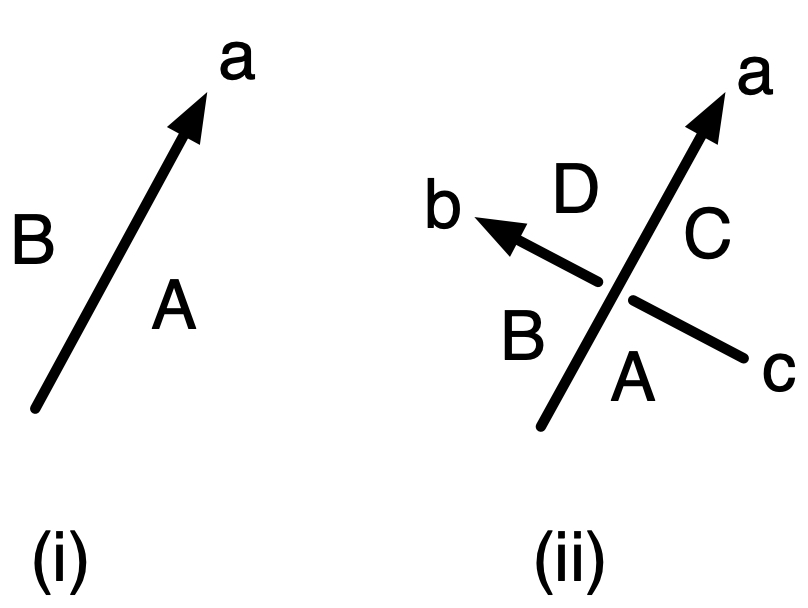}
\caption{(i) $f_{wd}(a)= A^{-1}B$; (ii) $f_{wd}(ab) = f_{wd}(ca)$}
\label{WD}
\end{center}
\end{figure}

In fact $f_{wd}$ is an isomorphism. Our construction of the inverse homomorphism $f_{dw}: \pi_{dehn} \to \pi_{wirt}$ uses ``integration," which we describe next.

Beginning in a region $R$, we travel along a path $\g$ to another region $R'$. As we do this, we build an element of $\pi_{wirt}$ by ``integration," successively appending the generators of $\pi_{wirt}$ (or their inverses) to the right, corresponding to the arcs of the diagram that we cross, as in Figure \ref{integrate}(i). We will denote the final element by $\int_\g \D$, and call it the result of \textit{integration} along $\g$. 
\begin{figure}[H]
\begin{center}
\includegraphics[height=2 in]{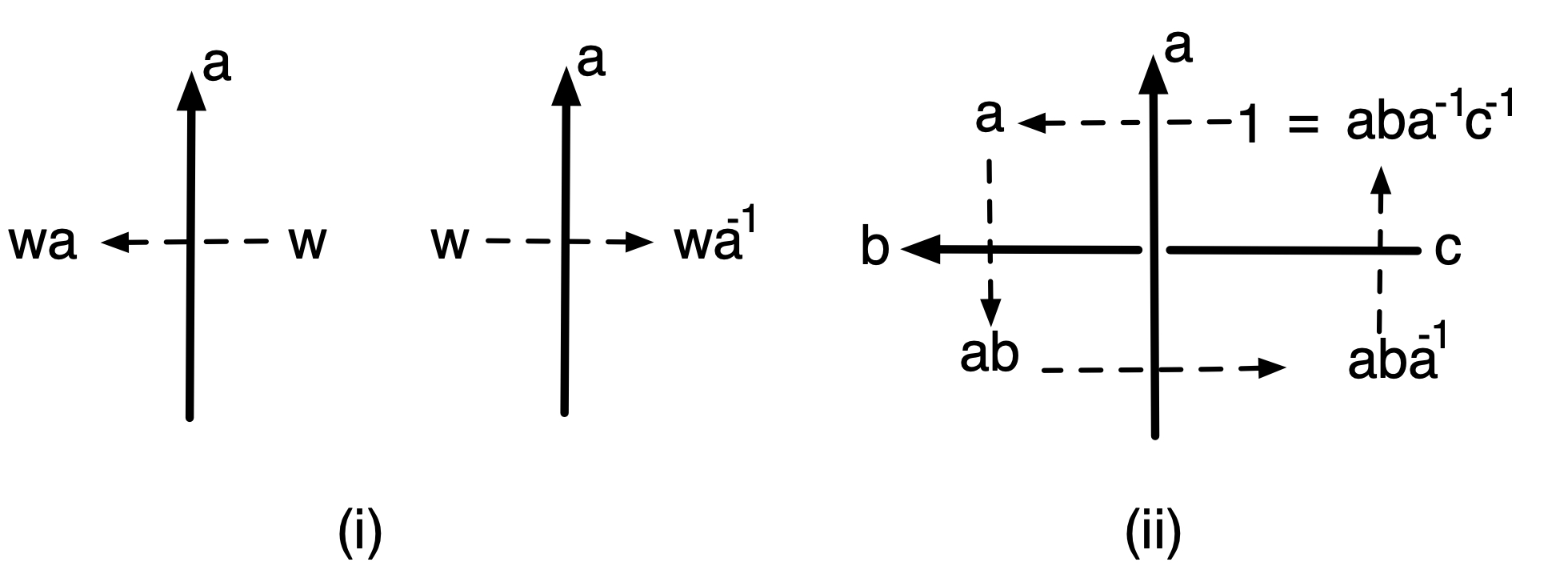}
\caption{(i) Integrating along a path; (ii) showing that $\int_\g \D = 1$ on a small loop.}
\label{integrate}
\end{center}
\end{figure}

We define $f_{dw}(R)$ to be  $\int_\g \D$, where $\g$ is any path from the base region $R_0$ to $R$. 
But is $f_{dw}$ well defined? Proving that is equivalent to proving that our integration is path independent. Given two paths with same initial  point in $R$ and final point in $R'$, a loop is formed from their concatenation, following one path and then going along the other in the opposite direction. The claim that integration is path independent is equivalent to the claim that the path integral around the loop is trivial. Figure \ref{integrate}(ii)) shows this for small loops about a crossing. Since the 2-sphere is simply connected, the verification for small loops about crossings implies the general claim. We leave the details to the energetic reader. 

We have shown that $f_{dw}$ is well defined on generators. To see that it induces a homomorphism on $\pi_{dehn}$, we must verify that it is trivial on a general Dehn relation $A^{-1}B = C^{-1}D$, as in Figure \ref{WD}(ii). If $A$ is sent to an element $w$, then $B$ maps to $wa$. Moreover, $C$ is sent to $wc^{-1}$ and $D$ is mapped to $wc^{-1}a$. Now $A^{-1}B$ and $C^{-1}D$ both map to the same value, $a$. The case of a left-hand crossing is similar. 

Finally, we check that $f_{wd}$ and $f_{dw}$ are inverses of each other. The verification is brief and we will leave it to the reader. 

We have proven the well-known result:

\begin{prop} If $\ell$ is an oriented link in $\R^3$, then $\pi_{wirt} \cong \pi_{dehn}$. \end{prop}

\begin{remark} The terms ``derivative" and ``integral" are used suggestively. But what do they suggest? We propose to think about a link diagram with arcs labeled by corresponding Wirtinger generators as a conservative vector field. Path integration produces labels of the regions by elements of $\pi_{wirt}$ that we associate with Dehn generators via $f_{dw}$. Thus the Dehn generator labeling might be viewed as a potential function, with the integral $f_{dw}(R)= \int_\g\D$ being the work done by the field as we move from the base region $R_0$ to $R$ along the path $\g$. Differentiating returns the original arc labeling. 
\end{remark}

\section{Links in Thickened Surfaces} Moving to the world of link diagrams on surfaces we find that much remains unchanged. Given an oriented link diagram $\D$ in a closed, connected orientable surface $S$ of genus $g >0$, we can again form the Wirtinger presentation of a group $\pi_{wirt}$ and also the Dehn presentation of a second group $\pi_{dehn}$, and the demonstration of invariance under Reidemeister moves is unchanged. The groups need no longer be isomorphic, and so we will call $\pi_{wirt}$ the \textit{Wirtinger 
link group} and $\pi_{dehn}$ the \textit{Dehn link group}. 

In order to describe the relationship between $\pi_{wirt}$ and $\pi_{dehn}$, we will again make use of integration. We will need a couple of facts about it, as surfaces of positive genus are more complicated than spheres. While the first is quickly proved using basic algebraic topology, a geometric argument is possible. The second statement is immediate from the definition.

\begin{lemma} \label{integrate} Assume that $\D$ is an oriented link diagram on a closed, connected orientable surface $S$, and $\g_1, \g_2$ are oriented paths in $S$.  \medskip

(i)  If $\g_1$ and $\g_2$  have the same endpoints and are homotopic rel boundary (that is, homotopic keeping endpoints fixed), then $\int_{\g_1}\D = \int_{\g_2}\D$.  \medskip

(ii) If the terminal point of $\g_1$ is the initial point of $\g_2$ and $\g_1\g_2$ is the concatenated path, then $\int_{\g_1\g_2}\D = \int_{\g_1}\D \int_{\g_2}\D$.

\end{lemma} 

We can define a homomorphism $f_{wd}: \pi_{wirt} \to \pi_{dehn}$  just as we did in the previous section, mapping any generator $a$ to $A^{-1}B$, where $A$ is the region to the right of the arc representing $a$ and $B$ is the region to the left. However, $f_{wd}$ is generally no longer injective. To see why we need to look closely at the surface $S$. 

We will visualize the surface $S$ of genus $g$ as a $2g$-gon with directed sides $x_1, y_1, \ldots, x_g, y_g$ identified in the usual way. Think of the bouquet of loops in $S$ as a coordinate system. Of course, there are other bouquets along which we could cut $S$ to get a $2g$-gon. We will say more about that in the last section. 

Without loss of generality we assume that the diagram $\D$ meets the curves $x_i, y_i$ in \textit{general position}, which means that $\D$ intersects them transversely and avoids the common point. Then each $x_i, y_i$ determines a word $r_i = \int_{x_i}\D,  s_i= \int_{y_i}\D$, respectively, by  integration, as illustrated in Figure \ref{peripheral}. The elements of $\pi_{wirt}$ that they determine are unchanged by Reidemeister moves. 

\begin{definition} The \textit{surface subgroup} of $\pi_{wirt}$, denoted by $\pi_{wirt}^S$, is the normal subgroup generated by $r_1, s_1, \ldots, r_g, s_g$. \end{definition}

\begin{figure}[H]
\begin{center}
\includegraphics[height=1.3 in]{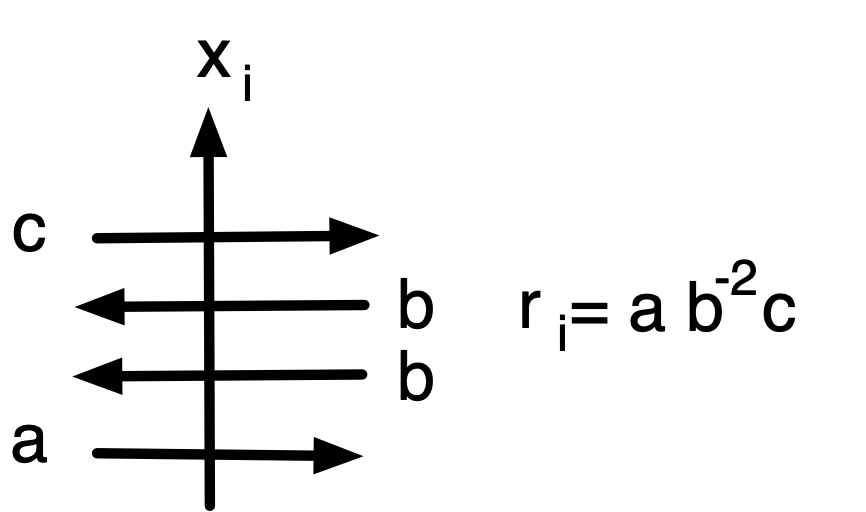}
\caption{Word $r_i$ determined by curve $x_i$}
\label{peripheral}
\end{center}
\end{figure}

\begin{lemma}\label{kernel} $\pi_{wirt}^S$ is in the kernel of $f_{wd}: \pi_{wirt} \to \pi_{dehn}$.
\end{lemma} 

\begin{proof} A relator $r_i$ has the form  $a_1^{\e_1} a_2^{\e_2} \cdots a_{2m}^{\e_{2m}},$
where $\e_1, \ldots, \e_{2m} \in \{\pm 1\}$.  Let $A_0, \ldots, A_{2m-1}, A_{2m}=A_0$ be the regions of the diagram that we encounter as we follow $x_i$. We must show that the image $f_{wd}(r_i)$ is trivial. This is clear since $f_{wd}(a_i^{\e_i}) = A_{i-1}^{\e_i}A_i$, for all $i$ and the two cases $\e_i =1, \e_i=-1$. Thus $f_{wd}(a_1^{\e_1}\cdots a_{2m}^{\e_{2m}}) = A_0^{-1}A_0 = 1$.

A similar argument applies to the relators $s_i$. 
\end{proof}

We come now to the lesson of our story. If we try to use integration to define $f_{dw}$ as we did in the previous section, then Lemma \ref{kernel} warns us that the result will not be well defined. Remember that path-independence of integration is equivalent to the requirement that the path integral around any closed loop is trivial. Integrating around $x_i, y_i$ generally produce nontrivial elements. However, if we replace $\pi_{wirt}$ by the quotient $\pi_{wirt}/\pi_{wirt}^S$, then path-independence is recovered. At the same time, we arrive at the relationship between $\pi_{dehn}$ and $\pi_{wirt}$.

\begin{theorem} \label{DW} \cite{By12} If $\ell$ is an oriented link in a thickened surface $S \times [0,1]$, then $\pi_{dehn} \cong \pi_{wirt}/\pi_{wirt}^S$. 
\end{theorem} 

\begin{proof} We begin by showing that integration is path-independent provided we take values in the quotient group $\pi_{wirt}/\pi_{wirt}^S$.  For convenience we will assume that the base region $R_0$ contains the common point $*$ of the loops $x_i, y_i$. Consider a loop $\g$ beginning and ending at $*$. 

We can write $\g$ up to homotopy fixing $*$ as a product $\g_1^{\e_1}\cdots \g_n^{e_n}$, where each
$\g_j \in \{x_1, y_1, \ldots, x_g, y_g\}$ and each $\e_j  \in \{\pm 1\}$. (This follows from the fact of algebraic topology that the fundamental group of $S$ is generated by the loops $x_i, y_i$. However, one can see this
directly by puncturing the $2g$-gon that forms the surface at some point not in $\g$, and then retracting the punctured $2g$-gon to its boundary.) Since each $\int_{x_i}\D, \int_{y_i}\D$ is in $\pi_{wirt}^S$, so is 
$\int_\g\D$ by Lemma \ref{kernel}.

By Lemma \ref{kernel}, the homomorphism $f_{wd}: \pi_{wirt} \to \pi_{dehn}$ induces a homomorphism 
$\bar f_{wd}: \pi_{wirt}/\pi_{wirt}^S \to \pi_{dehn}$. 
Define $\bar f_{dw}: \pi_{dehn} \to \pi_{wirt}/\pi_{wirt}^S$ by path integration, as we did in the previous section, but taking values in the quotient group  $\pi_{wirt}/\pi_{wirt}^S$. It is a simple matter to verify that 
the composition of $\bar f_{wd}$ and $\bar f_{dw}$ in either direction is the identity homomorphism.

 \end{proof}

\begin{remark}\label{DWremark} (i) Theorem \ref{DW} is the main result of \cite{By12}. 

(ii) Theorem \ref{DW} implies that $\pi_{dehn}$ does not depend upon the choice of base region $R_0$. 

(iii) If we do not choose a base region $R_0$, then the Dehn presentation that we get describes the free product $\hat\pi_{dehn} = \pi_{dehn} * \Z.$ The standard proof of this classical result for planar diagrams (see, for example, \cite{LS77}) can be adapted in the case of higher genus surfaces. \end{remark}

\begin{example} Our link diagrams $\D$ arise from the projection of $\S \times [0,1]$ onto $S$ from above; that is, overcrossing arcs correspond to larger values of the second coordinate. In this sense, $\pi_{wirt}$ is the \textit{upper Wirtinger link group}. If instead we project from below, then another diagram of $\ell$ is obtained, and the resulting Wirtinger group, the \textit{lower Wirtinger link group}, can be different when $S$ has positive genus. Consider the knot in Figure \ref{topbottom} viewed from the two perspectives. The reader can verify that the upper Wirtinger group has presentation 
$\langle a, b \mid aba=bab \rangle,$ while the lower Wirtinger group is infinite cyclic.

\begin{figure}[H]
\begin{center}
\includegraphics[height=2 in]{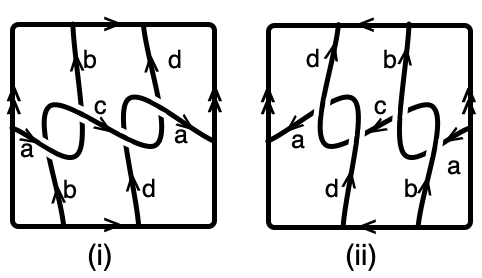}
\caption{Distinct upper and lower Wirtinger knot groups (cf. \cite{GPV00}, Fig. 7)}
\label{topbottom}
\end{center}
\end{figure}

A reason for the difference can be found in algebraic topology. Recall that for a link in $\R^3$ or $\S^3$, its Wirtinger link group is the fundamental group of the link complement. Choosing the basepoint of the fundamental group above a diagram results in an upper group presentation while placing it below results in the lower group presentation. Since each group is the fundamental group of the link complement, the upper and lower Wirtinger presentations describe the same group. 

For a link $\ell$ in $S \times [0,1]$, where $S$ is a  surface of arbitrary genus, the upper Wirtinger group can be seen to be $\pi_1((S \times [0,1] \setminus \ell) / S \times \{1\})$,  the fundamental group of $S\times [0,1] \setminus \ell$ with $S \times \{1\}$ coned to a point that serves as fundamental group basepoint.
Similarly, the lower group is  $\pi_1((S \times [0,1] \setminus \ell) / S \times \{0\})$. Less obvious is that the Dehn link group is the fundamental group of $S \times [0,1] \setminus \ell$ with both $S \times \{0\}$ and $S \times \{1\}$ coned to separate points, and hence the ``upper" and ``lower" Dehn link groups are the same (trivial in the above example). These facts were previously observed by N. Kamada and S. Kamada \cite{Ka20}. They will not be used here and are mentioned only the sake of motivation.

\end{example}

\section{Fox's Free Differential Calculus}\label{color}

In a series of papers beginning in 1953, R.H. Fox introduced the ``free differential calculus,"  a method for constructing  invariants for groups from presentations \cite{Fo54}. Although inspired by homology calculations in covering spaces, it is a completely combinatorial method.

Let $H$ be a group. We will make use of the group ring $\Z[H]$.
It consists of all finite linear combinations $\sum n_h h$, where each $n_h\in \Z$ and $h \in H$. 
Addition is defined coordinate-wise by: 
$$\sum m_h h + \sum n_h h = \sum (m_h + n_h) h,$$
while multiplication is given by:
$$(\sum m_h h)(\sum n_{h} h) = \sum(\sum_{h=kk'} m_k n_{k'}) h.$$
Note that $H$ is embedded in $\Z[H]$ in a natural way. 
We can think of elements of $\Z[H]$ as a linearization of $H$. 

The partial derivative $\partial/\partial x_i$ is a  homomorphism from $\Z[F]$ to itself, defined by:

$${\partial x_i \over \partial x_j} = \d_{ij}, \quad  {\partial x_i^{-1} \over \partial x_j} = -\d_{ij} x_i^{-1}$$
$${\partial (uv )\over \partial x_j} ={ \partial u \over \partial x_j }+ u {\partial v \over \partial x_j}.$$
The last equation is particularly useful when $v$ is the last symbol $x_i^{\pm 1}$ of a word.

%

Given a presentation $\<x_1, \ldots, x_n \mid R_1, \ldots, R_m \>$ of a group $\pi$, its \textit{Jacobian matrix} $J$ is the $m \times n$ matrix with $ij$th entry $\partial R_i/\partial x_j$. (If a relation $R=S$ appears in a presentation instead of a relator, then we can 
take a partial derivative of each side and subtract the results or we can form the relator $RS^{-1}$ and take its partial derivative. The outcomes will be the same.)

How can we build invariants of a presented group $\pi$ using the free differential calculus? Begin by choosing a homomorphism $\phi$ from $\pi$ to an abelian group $H$. It extends to a homomorphism 
$\phi: \Z[\pi] \to \Z[H]$. Applying $\phi$ to each coefficient of the Jacobian matrix $J$, we get the \textit{specialized Jacobian matrix} $J^\phi$. The quotient $\Z[H]^n/J^{\phi}\Z[H]^n$, the \textit{cokernel} of $J^\phi$, describes a module over the the group ring $\Z[H]$. It has generators $x_1, \ldots, x_n$, and relations corresponding to the rows of the matrix $J^\phi$. The $i$th row is the relation $\sum_j (\partial r_i/\partial x_j) x_j= 0$.  By \cite{Fo54},  the module is independent of the presentation of the group $\pi$. The strategy of the proof is to show that Tietze transformations change the Jacobian matrix only up to elementary transformations.

Let $\pi= \pi_{wirt}$ be the Wirtinger group of a link  in a thickened surface $S$, and consider the homomorphism $\phi$ from $\pi$ to the infinite cyclic group $\langle t \rangle$, mapping each generator to $t$. The entries of $J^\phi$ are integral polynomials in the variables $t, t^{-1}$. When $S$ is the 2-sphere, the module presented is well known to knot theorists: it is the first homology group of the infinite cyclic cover of the link complement. 

\begin{example} \label{trefoilex} Consider the Wirtinger presentation $\pi= \langle a, b, c \mid ab=ca, bc=ab, ca=bc\rangle$ of the trefoil knot in $\S^3$ (Figure \ref{trefoil}). (Here it is convenient to denote generators by $a, b, c$, and avoid confusion with previously defined loops in $S$.) The reader can verify that
$$J = \begin{pmatrix} 1-c & a & -1\\ -1&1-a &b \\ c& -1 & 1-b \end{pmatrix}.$$

Let $\phi: \pi \to \langle t \rangle$ be the homomorphism that maps each generator to $t$. Then 

$$J^\phi = \begin{pmatrix} 1-t & t & -1\\ -1&1-t &t \\ t& -1 & 1-t \end{pmatrix}.$$

\begin{figure}[H]
\begin{center}
\includegraphics[height=1.3 in]{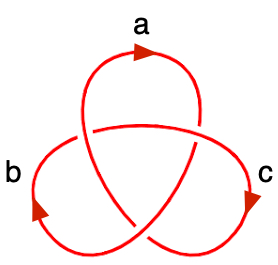}
\caption{Diagram of a trefoil knot with Wirtinger generators}
\label{trefoil}
\end{center}
\end{figure}

\end{example}

\section{The Dehn Coloring Group} \label{DCG}

Let $\nu$ be the homomorphism of $\pi= \pi_{wirt}$ to the multiplicative group $H=\{\pm 1\}$ of order 2,  sending each Wirtinger generator to $-1$.  The specialized Jacobian matrix $J^\nu$ has entries in $\Z[\pm 1]\cong \Z$. The partial derivatives of a Wirtinger relation $ab=ca$ (Figure \ref{WD}) contribute to $J^\nu$ a row corresponding to the relation $a-b=c-a$. Rewritten as $2a=b+c$, it is the well-known Fox coloring condition for arcs of a diagram, as in Figure \ref{fox}\footnote{Let $p$ be a prime. If we regard $a, b$ and $c$ as elements of $\Z/p$ (``colors"),  then the condition says that the  color of any overcrossing arc of the diagram is equal to the sum of the colors of the two arcs below it.}. (The reader should observe that $J^\nu$ can be recovered from $J^\phi$ in Example \ref{trefoilex} by replacing $t$ with $-1$.)
	
\begin{figure}[H]
\begin{center}
\includegraphics[height=1.3 in]{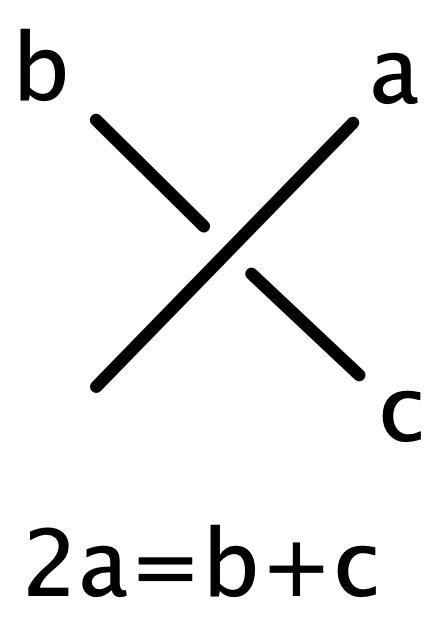}
\caption{Fox coloring rule}
\label{fox}
\end{center}
\end{figure}

Any link diagram in the 2-sphere (or plane) can be \textit{checkerboard shaded}, some of its regions shaded so that whenever two regions share a common arc, exactly one of them is shaded. If a diagram admits a checkerboard shading, then it admits exactly two distinct checkerboard shadings. 

What about diagrams in surfaces of higher genus? A diagram of a link $\ell$ in a surface $S$ can be checkerboard shaded if and only if $\ell$ represents the trivial element of $\Z/2$-homology $H_1(S;\Z/2)$. This condition is equivalent to the requirement that the diagram meets each loop $x_i, y_i$ in an even number of points. Proving this is a nice exercise.

Consider a checkerboard shaded link diagram in a surface. The homomorphism $\bar \nu = \nu \circ  f_{dw}$ maps all unshaded Dehn generators of $\D$ to the same element of $\{\pm 1\}$, and all shaded generators to the other. The result of applying Fox calculus to a Dehn relation such as $A^{-1}B = C^{-1} D$ (Figure \ref{WD}(ii)) is $A+B = C+D$, which we call the \textit{Dehn coloring condition} for $\D$ (see Figure \ref{Dehn coloring}). The calculation depends only on the fact that $A, D$ map to the same value while $B, C$ map to the other. See \cite{CSW14'} for additional details. 

\begin{figure}[H]
\begin{center}
\includegraphics[height=1.7 in]{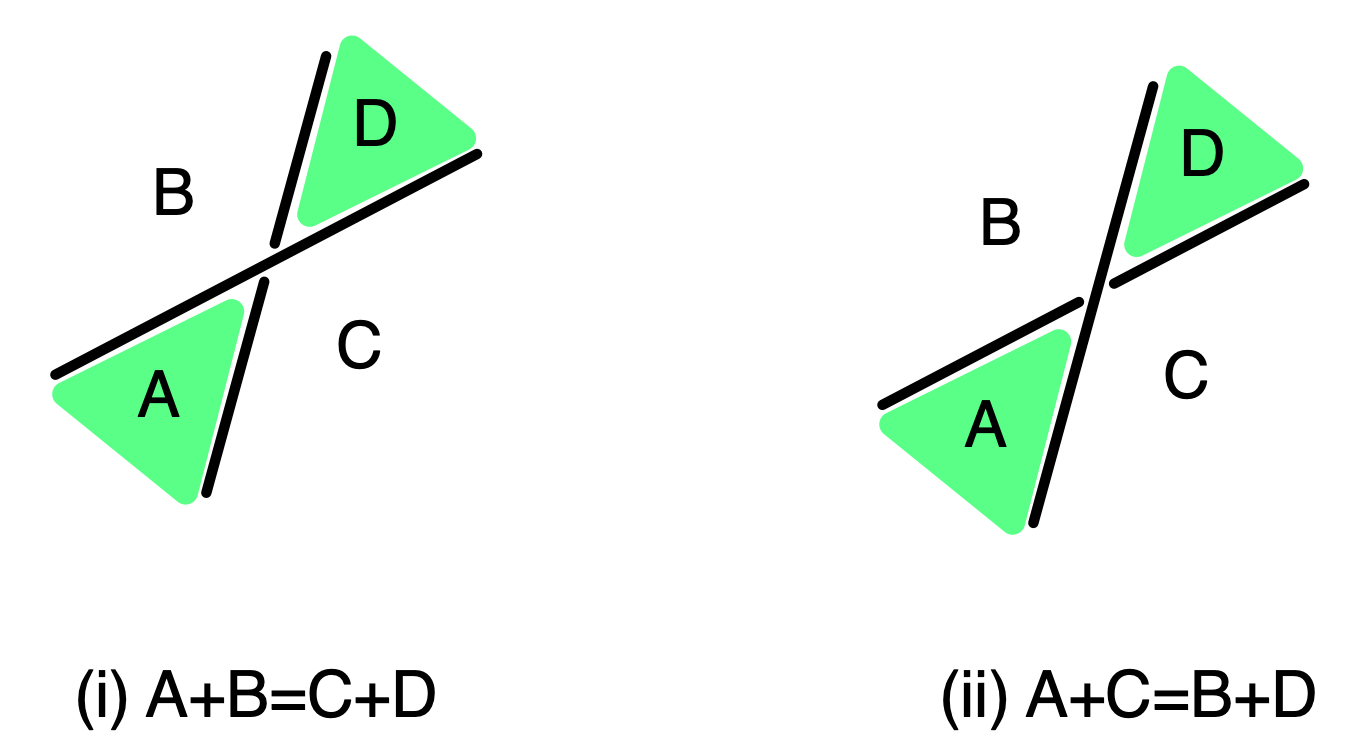}
\caption{Dehn coloring condition}
\label{Dehn coloring}
\end{center}
\end{figure}

\begin{definition} The \textit{Dehn coloring group} $\C$ of a link $\ell$ in a thickened surface $S \times [0,1]$ is the cokernel $\Z^n/J^{\bar \nu}\Z^n$, where $J^{\bar \nu}$ is the Jacobian $n \times n$ matrix of a Dehn presentation for the group of $\ell$ and $\bar \nu = \nu \circ  f_{dw}$. 
\end{definition}

\begin{remark} \label{homology} When $S = \S^2$ it is well known that the Dehn coloring group $\C$ is isomorphic to $H_1(M_2; \Z)\oplus \Z$, where $M_2$ is the 2-fold cyclic cover of $\S^3 \setminus \ell$ corresponding to the homomorphism that maps each meridian to a generator of $\Z/2$. \end{remark}

Given any link diagram $\D$ in a $S$, we can apply Reidemeister moves to assure that all regions are contractible. Then if the diagram is checkerboard shaded, we can construct a graph $G$, embedded in $S$ and unique up to isotopy, with a vertex in each shaded region and an edge through each crossing joining a pair shaded regions. (An edge is allowed to join a vertex to itself.) Such a graph is called a \textit{Tait graph} in honor of the nineteenth-century Scottish pioneer of knot theory (and golf enthusiast) Peter Guthrie Tait.  For each edge $e$ of $G$ we assign a weight $w_e= \pm 1$ according to the type of crossing involved, as in Figure \ref{medial4}. In order to avoid notational clutter, unlabeled edges are assumed to have weight $+1$. 
We will use the Tait graph to determine all of the Dehn coloring relations. Note that unshaded regions of $\D$ become faces of $G$. We will refer to the Dehn generators corresponding to shaded and unshaded regions of $\D$ as \textit{vertex generators} and \textit{face generators}, respectively.

\begin{figure}[H]
\begin{center}
\includegraphics[height=1.5 in]{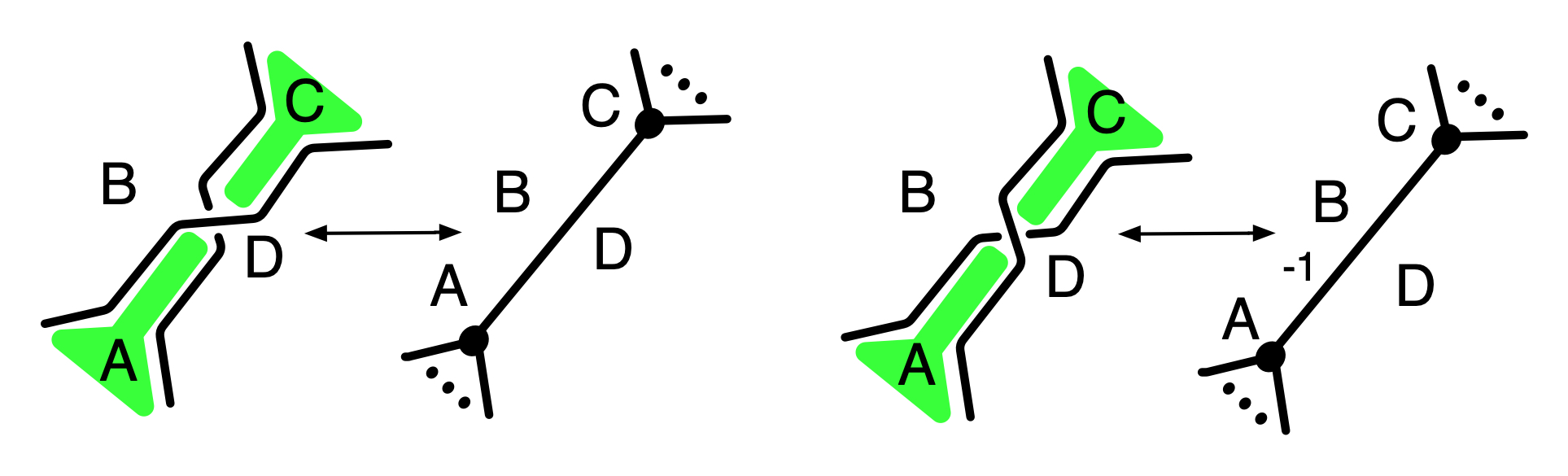}
\caption{Constructing a Tait graph from a checkerboard shaded diagram. Shaded (resp. unshaded) generators of $\D$ become vertex (resp. face) generators of $\D$.}
\label{medial4}
\end{center}
\end{figure}

Recall that there are two checkerboard shadings of $\D$. If we use the other shading, then we get a Tait graph $G^*$ that is dual to $G$. Each edge $e^*$ of $G^*$ meets an edge $e$ of $G$ transversely in a single point. The product $w_e w_{e^*}$ of weights is  $-1$.

The \emph{adjacency matrix} of any edge-weighted graph $G$ with vertices $v_1, \ldots, v_n$ is the $n \times n$ matrix 
$A = (a_{i, j})$ such that $a_{i, j}$ is the sum of the weights of edges between $v_i$ and $v_j$.
An edge joining a vertex $v_i$ to itself is counted twice, so it contributes $\pm 2$ to $a_{i,i}$. Define $\d= (\d_{i, j})$ to be the $n \times n$ diagonal matrix with $\d_{i,i}$ equal to the sum of the weights of edges incident on $v_i$, again counting loops twice.

\begin{definition} \label{lapdef} The \emph{Laplacian matrix} $L_G$ of a finite graph $G$ is $\d - A$. The  \emph{Laplacian group} ${\L}_G$ is the cokernel $\Z^n/ L_G\Z^n$. 
\end{definition}

Using Reidemeister moves it can be shown that the pair $\{\L_G, \L_{G^*}\}$ is an invariant of the link $\ell$. See \cite{SW19'} for details.

The reader might wonder why we have introduced yet another group. The answer is that there is 
a relationship between the Laplace group $\L_G$ and the Dehn coloring group $\C$. We see the relationship by using Dehn relations to eliminate face generators of $\C$.

Around any vertex $v$ of $G$, write the Dehn coloring conditions for all of the adjacent edges, always putting the face generator corresponding to the region to the left of the edge, as viewed from $v$, on the left side of the equation. (This puts $v$ on the right side of the equation if the edge carries a negative weight.) Define $R_v$ to be the sum of the relations. The face generators cancel in pairs,  and so $R_v$ is a relation in the vertex generators. It is not difficult to see that $R_v$ is in fact the relation in $\L_G$ associated to $v$. 

As an example, consider  Figure \ref{holonomy}. Here 
$$v+u_1 = v_1 + u_2$$
$$v+u_2 = v_2 + u_3$$
$$v_3 + u_3 = v + u_1$$
and so $R_v$ is the relation:
$$v+ v_3 = v_1 +v_2,$$
which can be written as the Laplacian group relation: 
$$v = v_1 + v_2 - v_3.$$

\begin{figure}[H]
\begin{center}
\includegraphics[height=2 in]{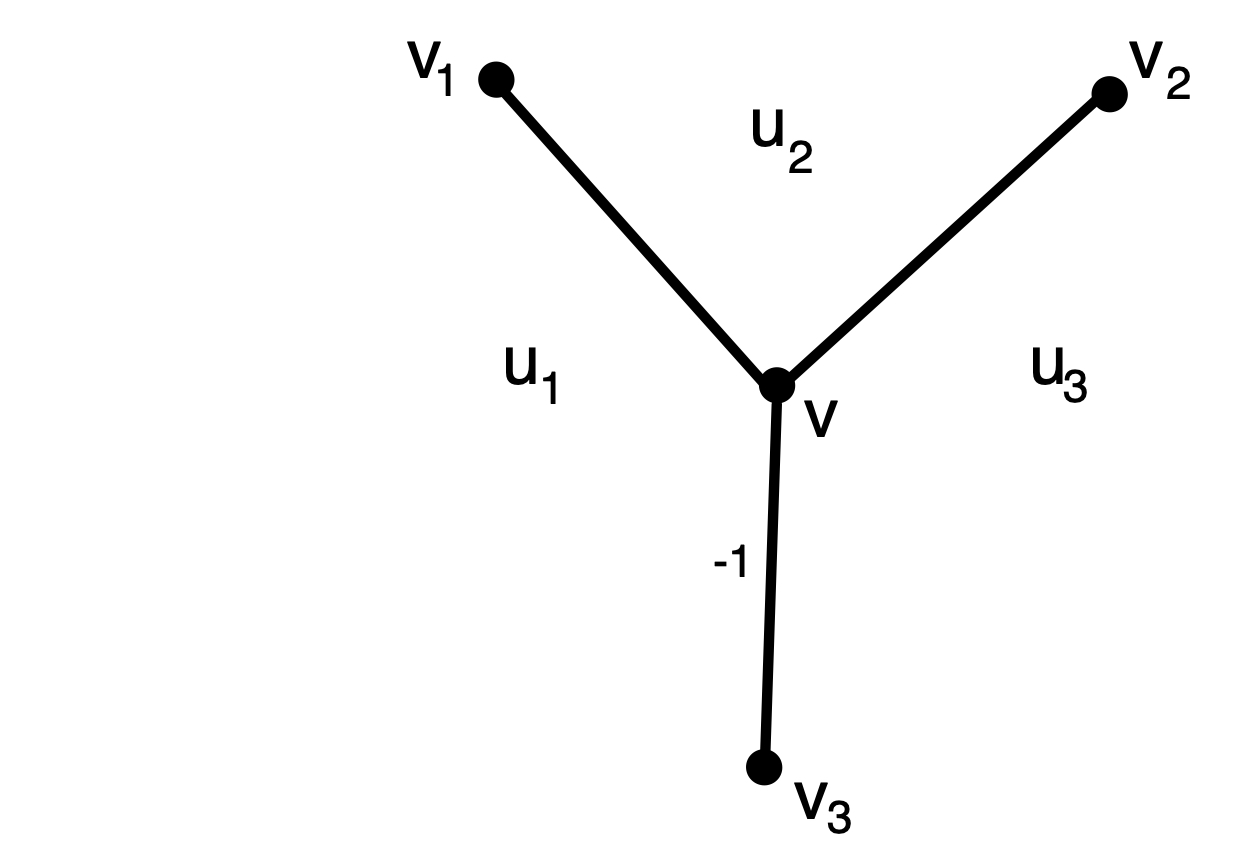}
\caption{Dehn coloring conditions produce the relation $R_v:$ $v=v_1+v_2-v_3$}
\label{holonomy}
\end{center}
\end{figure}

We will define $\C^s$ to be the subgroup of $\C$ generated by the vertex generators. Then $R_v$ is a relation of $\C^s$. As we will see, such relations form a sufficient set for $\L_G$ when $S = \S^2$. For surfaces of positive genus, the additional needed relations are easy to describe. They are the result of rewriting as we go around the loops $x_i, y_i$. 

Every Dehn relation can be written in the form $v_1-v_2= u_1 -u_2$, for some vertex generators $v_1, v_2$ and face generators $u_1, u_2$. Any relation in $\C^s$ is a sum of Dehn relations in which the face generators cancel in pairs. Such relations correspond to circuits in the dual graph $G^*$, and hence to closed paths in $S$. The relations $R_v$, as $v$ varies over all vertices of $G$, generate the relations arising from contractible closed paths. (The proof is similar to the suggested argument in Section \ref{IntDiff} for showing that integration is path independent.) Consequently, since every closed path in the 2-sphere is contractible, $\C^s \cong {\L}_G$ when $S=\S^2$. 

Now let's replace the checkerboard shading of our link diagram with the other checkerboard shading. This reverses the roles of shaded and unshaded regions, and it replaces the Tait graph $G$ with the its dual $G^*$. If we define $\C^u$ to be the subgroup of $\C$ generated by face generators, then we find that $\C^u \cong {\L}_{G^*}$ when $S=\S^2$. 

Knot theorists recognize $L_G$ and $L_{G^*}$ as Goeritz matrices for the link $\ell$ described by the diagram. Both ${\L}_G$ and ${\L}_{G^*}$ are isomorphic to $H_1(M_2; \Z),$ where $M_2$ is the 2-fold cover of $\S^3$ branched over the $\ell$ (see Remark \ref{homology}). In particular, ${\L}_G \cong {\L}_{G^*}$. We will give an alternative proof of this fact, independent of algebraic topology.

\begin{prop}\label{iso} If $\D$ is a link diagram in $\S^2$, then the Laplacian groups ${\L}_G, {\L}_{G^*}$ are isomorphic. 
\end{prop}

\begin{proof} Consider the short exact sequence of abelian groups
$$0 \to \C^s \hookrightarrow \C \to \C/\C^s \to 0.$$
The relation $v_1-v_2= u_1 -u_2$ in $\C$ becomes $0 = u_1-u_2$ in the quotient $\C/\C^s$, and hence face generators become equal whenever they share an edge in the graph $G^*$. Since we assume that all regions of $\D$ are contractible, $G^u$ is connected. The quotient is infinite cyclic, and the short exact sequence splits. Hence $\C \cong \C^s \oplus \Z$. The same argument applied to $\C^u$ shows that $\C \cong \C^u \oplus \Z$. It follows from the structure theorem for finitely generated abelian groups that $\D^s$ and $\D^u$ are isomorphic. Since
 $\C^s$ and $\C^u$ are isomorphic to ${\L}_G$ and ${\L}_{G^*}$, respectively, the proof is complete. 
\end{proof} 

\begin{remark} Proposition \ref{iso} can be proved yet a third way, using the fact that the Tait graph $G$ can be converted to its dual $G^*$ by a sequence of Reidemeister moves \cite{YK57}. (See also \cite{SW19'}.)
\end{remark} 

\section{The Dehn Coloring Module}  The Dehn coloring group $\C$ and the Laplacian group $\L_G$ associated to a checkerboard shaded diagram $\D$ of a link can be made stronger invariants, as done in \cite{SW19'}, using homological information from $S$. Then $\C$ and $\L_G$ become modules over the group ring of $H_1(S; \Z)$. 

We think of $H_1(S; \Z)\cong \Z^{2g}$ as the multiplicative abelian group freely generated by $x_1, y_1, \ldots, x_g, y_g$. Recall that these generators are represented by a bouquet of simple oriented loops in $S$ (denoted by the same symbols). The universal abelian cover $\tilde S$ of $S$ has deck transformation group $A(\tilde S)$ that is isomorphic to $H_1(S; \Z)$. The Dehn coloring module and Laplacian module that we will define are modules over the ring $\La = \Z[x_1^{\pm 1}, y_1^{\pm 1}, \ldots, x_g^{\pm 1}, y_g^{\pm 1}]$ of Laurent polynomials. In this section we will let $\C$ and $\L_G$ denote these modules.

%
%
Again, we view $S$ as a $2g$-gon with sides $x_1, y_1, \ldots, x_g, y_g$ identified. We may also think of the $2g$-gon as a fundamental region of the universal abelian  cover $\tilde S$.  In order to define the  module $\C$,
label regions of the diagram by $A, B,C, \ldots$. If a region $A$ of $\D$ is divided into several subregions in the $2g$-gon, then we choose one subregion to receive the label $A$. Assume that $A'$ is another subregion. If $w$ is the element of $A(\tilde S)$ such that $A$ and $A'$ are in the same region of $\tilde S$, 
then replace $A'$ by the label $wA$. (An example is seen in Figure \ref{WDex}.) The \textit{Dehn coloring module} $\C$ is a module over the group ring $\La$ with generators $A, B, C, \ldots$. Defining relations are as for the Dehn coloring group (Figure \ref{Dehn coloring}).  We will once more refer to shaded and unshaded generators.

 The \textit{Laplacian matrix} is given by $L_G=\delta-A$ where $\delta$ is as before, but the adjaceny matrix $A$ now has coefficients in $\La$.  Edge weights $\pm 1$ are replaced by $\pm w$, where $w$ is the element of $A(\tilde S)$ determined by following the edge from $v_i$ to $v_j$. (See Example \ref{example}.)
The \textit{Laplacian module} $\L_G$ is the cokernel of the matrix.

It is reassuring to note that if all generators $x_k, y_k$ are set equal to $1$, then $\C$, $L_G$ and $\L_G$ become the Dehn coloring group, Laplacian matrix  and Laplacian group, respectively, of the previous section.  Using Reidemeister moves we see that $\C$ and $\L_G$ are link invariants. The reader can verify this or see the proof given in \cite{SW19'}.

The \textit{Laplacian polynomial} $\De _G$ is defined to be the module order of $\L_G$, which is found as the determinant of $L_G$.
When $G$ is replaced by the dual graph $G^*$, we obtain the module order $\De_{G^*}$. The pair $\{\De_G, \De_{G^*}\}$ is an invariant of the link. While the two polynomials are not generally equal (see \cite{SW19'}) we have the following. 

\begin{prop} \label{same} Let $\ell$ be a checkerboard shaded link diagram in a  torus. If $G, G^*$ are its Tait graphs, then $\De_G = \De_{G^*}$.
\end{prop}

\begin{proof} As in Section \ref{color}, we let $\C^s$ be the submodule of $\C$ generated by shaded generators. 
The argument that follows Definition \ref{lapdef} can be modified to show that relations in $\C^s$ are obtained by eliminating unshaded generators along null-homologous closed paths in $S$. In particular, the relations $R_v$ corresponding to rows of $L_G$ arise from small loops encircling vertices of $G$. (To see this, it helps to view the $2g$-gon within the universal abelian cover $\tilde S$.) 
Since $S$ is assumed to be a torus, all null-homologous closed paths  are contractible. The relations $R_v$ suffice to generate the relations of $\C^s$. Hence $\C^s \cong \L_G$. 

Consider the effect on the module $\C$ of setting all shaded generators equal to $0$. Using Dehn relations, we find that all of the unshaded generators are equal; that is, $\C/\C^s \cong \Z$.
We have the short exact sequence: 

\begin{equation} 0 \to \C^s \to \C \to \Z \to 0 \end{equation} 

The module order $\De_0(\C)$  is equal to the product of the orders of $\C^s$ and $\Z$.  (See \cite{Mi68}, for example). The module order of $\C^s$ is $\De_G$ while that of $\Z$ is 1. Hence $\De_0(\C) = \De_G$. 

Replacing $G$ with its dual (or equivalently, reversing the checkerboard shading), yields $\De_0(\C) = \De_{G^*}$. Hence $\De_G = \De_{G^*}.$
\end{proof} 

\begin{example} \label{example}  Consider the diagram $\D$ of a 3-component link in the thickened torus that appears in Figure \ref{WDex}(i) with generators indicated for the Wirtinger group $\pi_{wirt}$. 

\begin{figure}[H]
\begin{center}
\includegraphics[height=4 in]{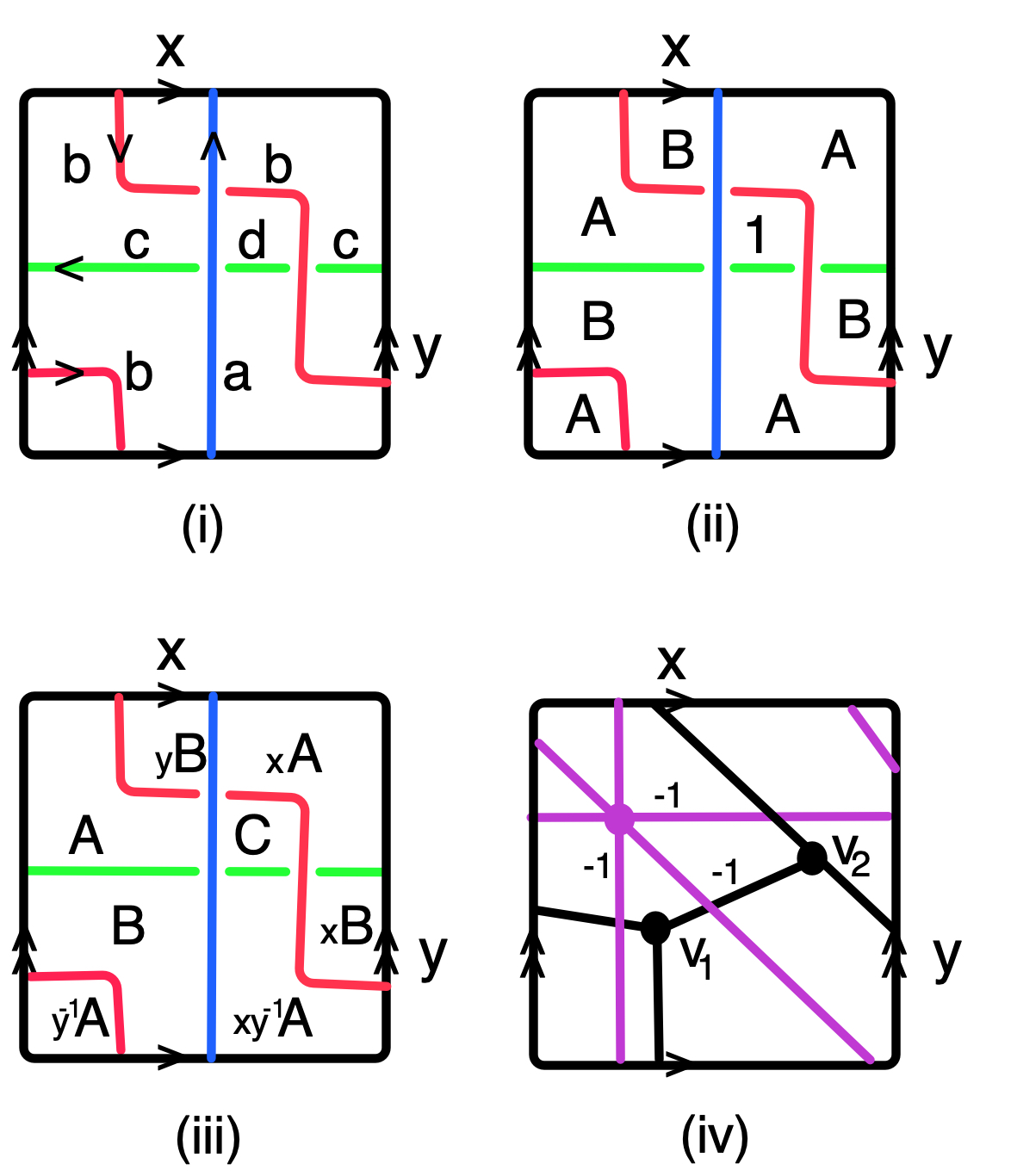}
\caption{3-component link: (i) generators for $\pi_{wirt}$; (ii) generators for $\pi_{dehn}$; (iii) module generators for $\C$; (iv) Tait graphs $G$ (black) and $G^*$ (purple). }
\label{WDex}
\end{center}
\end{figure}

One checks that 
$$\pi_{wirt} \cong \langle a, b, c, d \mid ab=ba,\ ac =da, \ bc=db \rangle.$$ 

The group $\pi_{wirt}$ is more easily recognized by eliminating a generator and introducing a new generator via Tietze transformations. We use the last relation, written as $d = b c b^{-1}$, to eliminate $d$. Then introduce $e = a^{-1}b$, and eliminate $b$. Consequently, 
$$\pi_{wirt} \cong \langle a, c, e \mid ae=ea,\ ce=ec \rangle,$$
which is the free product of two copies of $\Z^2$ amalgamated over the infinite cyclic subgroup generated by $e$. \smallskip

Theorem \ref{DW} implies that $\pi_{dehn}$ is isomorphic to the quotient of $\pi_{wirt}$ by the 
relations $a=b, b=c$, and hence $\pi_{dehn} \cong \Z$. 
We can see this directly. Using  
the generators of Figure \ref{WDex}(ii), we have:
$$\pi_{dehn} \cong \langle A, B \mid A =A^{-1}B, \ B^{-1}A=A^{-1}, \ A^{-1}B=A \rangle \cong \langle A, B \mid B = A^2 \rangle \cong \langle A \mid \, \rangle \ \cong \Z.$$

With the labeled generators of Figure \ref{WDex}(iii), we obtain a presentation of the Dehn coloring module:

$$\C\cong \langle A, B, C \mid  A + C  = B + xy^{-1} A,\  yB + xA= A + C,\ C + xA = xy^{-1} A + x B  \rangle.$$
Its module order is
$$\De_0(\C) = 2-x-x^{-1}+y+y^{-1} - x y^{-1}-x^{-1}y.$$

By Proposition \ref{same} the module order $\De_0(\C)$ agrees with both $\De_G$ and $\De_G^*$. 
This is easily verified: the polynomial $\De_G$ is the determinant of 
$$\begin{pmatrix} 1 & 1-x^{-1} -y^{-1}  \\ 1-x -y& 1 \end{pmatrix}.$$ To see the first row of the matrix, for example, note that vertex $v_1$ in $G$ is incident to $x^{-1} v_2, y^{-1}v_2$ by edges of weight $1$, and it is incident to $v_2$ by an edge of weight $-1$. 

Similarly, $\De_{G^*}$ is the determinant of the $1 \times 1$ matrix 
$$(2-x-x^{-1}+y+y^{-1}-xy^{-1}-x^{-1}y)$$.
\end{example}

\begin{remark} (i) The modules $\L_G$ and $\L_{G^*}$ in Example \ref{example} are easily seen to be isomorphic. (Simply eliminate one of the generators from the presentation for $\L_G$.) For arbitrary links in a thickened torus, however, the two modules need not be isomorphic. See Example 3.8 of \cite{SW19'}. \smallskip

(ii) The graph $G$ in Example \ref{example} is sometimes called a ``theta graph." Note that one edge has weight $-1$. By changing the location of the weight to different edges, we obtain three theta graphs $G_1, G_2, G_3$ corresponding to 3-component links $\ell_1, \ell_2, \ell_3$ (respectively) in the thickened torus $S$. As discussed in \cite{SW19'}, each of the links can be transformed into any other by a homeomorphism of $S\times I$, but they cannot be transformed by isotopy. The latter claim is seen by computing the three polynomials $\De_{G_1}, \De_{G_2}, \De_{G_3}$, which are easily seen to be different. \end{remark} 

(iii) The Dehn coloring group and the Laplacian group are invariants of the link up to homeomorphism of the thickened surface. On the other hand, the Laplacian module and polynomial require a homology basis $x_i, y_i$ for $S$. Such a basis acts in like a coordinate system. With it we can compare links in $S$. 
However, if we regard $\L_G$ and $\De_G$ up to symplectic change of basis, then they become invariants that are independent of the choice of basis. See \cite{SW19'} for details.

\bigskip

\ni Department of Mathematics and Statistics,\\
\ni University of South Alabama\\ Mobile, AL 36688 USA\\
\ni Email: \\
\ni  silver@southalabama.edu\\
\ni swilliam@southalabama.edu

\end{document}